\documentclass[12pt]{amsart}
\usepackage{amsmath,amssymb,amsfonts,amsxtra,verbatim}

\usepackage{times}

\usepackage{amssymb}
\usepackage{amsmath}
\usepackage{enumerate}

\makeatletter
\@namedef{subjclassname@2010}{%
  \textup{2010} Mathematics Subject Classification}
\makeatother

\newtheorem{thm}{Theorem}[section]

\newtheorem{lemma}[thm]{Lemma}

\theoremstyle{plain}
\numberwithin{equation}{section}

\theoremstyle{remark}

\newcommand{\Comment}[1]{}
\newcommand{\rbr}[1]{\left( {#1} \right)}

\newcommand{\cbr}[1]{\left\{ {#1} \right\}}

\newcommand{\norm}[1]{\left\|#1\right\|}
\newcommand{\eq}[1]{(\ref{#1})}

\newcommand{\bcbr}[1]{\Bigg\{ {#1} \Bigg\}}

\usepackage{environ}
\NewEnviron{as}{\begin{align}\begin{split}\BODY\end{split}\end{align}}
\NewEnviron{as*}{\begin{align*}\begin{split}\BODY\end{split}\end{align*}}

\def\CC{\mathbb{C}}
\def\RR{\mathbb{R}}

\def\ZZ{\mathbb{Z}}
\def\NN{\mathbb{N}}

\begin{document}
\title{Explicit Salem sets and applications to metrical Diophantine approximation}
\author{Kyle Hambrook}

\begin{abstract}
Let $Q$ be an infinite subset of $\ZZ$, let $\Psi: \ZZ \rightarrow [0,\infty)$ be positive on $Q$, and let $\theta \in \RR$.
Define 
$$
E(Q,\Psi,\theta) = \{ x \in \RR : \| q x - \theta \| \leq \Psi(q) \text{ for infinitely many $q \in Q$} \}.
$$
We prove a lower bound on the Fourier dimension of $E(Q,\Psi,\theta)$. This generalizes theorems of Kaufman and Bluhm and yields new explicit examples of Salem sets. We give applications to metrical Diophantine approximation, including determining the Hausdorff dimension of $E(Q,\Psi,\theta)$ in new cases. We also prove a higher-dimensional analog of our result.
\end{abstract}

\maketitle

\section{Main Result}\label{intro}

For $x \in \RR^d$, write $|x| = \max_{1 \leq i \leq d} |x_i|$ and $|x|_2 = (\sum_{i=1}^{d} |x_i|^2)^{1/2}$. 
For $x \in \RR$, $\norm{x} = \min_{k \in \ZZ}|x-k|$ is the distance from $x$ to the nearest integer. 
If $A$ is a finite set, $|A|$ is the cardinality of $A$. 
The expression $X \lesssim Y$ stands for ``there is a constant $C > 0$ such that $X \leq CY$.'' 
The expression $X \gtrsim Y$ is analogous.
The expression $X \approx Y$ means ``there are constants $C > c > 0$ such that $cY \leq X \leq CY$.'' 

Let $Q$ be an infinite subset of $\ZZ$, let $\Psi: \ZZ \rightarrow [0,\infty)$ be a function with $\Psi(q) > 0$ for all $q \in Q$, and let $\theta \in \RR$.
Define $E(Q,\Psi,\theta)$ to be the set of all $x \in \RR$ such that
$$
\| q x - \theta \| \leq \Psi(q) \text{ for infinitely many $q \in Q$}.
$$
We will always assume $\Psi$ is bounded. Since $\norm{x} \leq 1/2$ for all $x \in \RR$, assuming $\Psi$ is bounded results in no loss of generality. We will also always assume $\Psi(0)=1$. This assumption is imposed only to avoid tedious notation. Since redefining $\Psi$ at finitely many points does not change the set $E(Q,\Psi,\theta)$, assuming $\Psi(0)=1$ results in no loss of generality. 

For $M > 0$, define 
\begin{align*}
Q(M) &= \cbr{q \in Q : M/2 < |q| \leq M }, \\
\epsilon(M) &= \min_{q \in Q(M)} \Psi(q).
\end{align*}
The main result of this paper is the following theorem.

\begin{thm}\label{main-thm-1}
Suppose there is a number $a \geq 0$, an increasing function \mbox{$h:(0,\infty) \rightarrow (0,\infty)$}, and an unbounded set $\mathcal{M} \subseteq (0,\infty)$ such that
\begin{align}\label{main-thm-1 e1}
|Q(M)|\epsilon(M)^{a} h(M) \geq M^{a} \quad \forall M \in \mathcal{M}.
\end{align}
Then there is a Borel probability measure $\mu$ supported on $E(Q,\Psi,\theta)$ such that
\begin{align}\label{main-thm-1 e2}
|\widehat{\mu}(\xi)| \lesssim |\xi|^{-a} \exp\rbr{\frac{\ln |\xi|}{\ln \ln |\xi|}} h(4|\xi|) \quad \forall \xi \in \RR, |\xi| > e.
\end{align}
\end{thm}

We also have a higher-dimensional version of Theorem \ref{main-thm-1}. 

Let $m, n \in \NN$, let $Q$ be an infinite subset of $\ZZ^n$, let $\Psi: \ZZ^n \rightarrow [0,\infty)$ be a function with $\Psi(q) > 0$ for all $q \in Q$, and let $\theta \in \RR^m$. 
Define $E(m,n,Q,\Psi,\theta)$ to be the set of all points 
$$
(x_{11}, \ldots, x_{1n}, \ldots, x_{m1}, \ldots, x_{mn}) \in \RR^{mn}
$$ 
such that
$$
\max_{1 \leq i \leq m} \| \sum_{j=1}^{n} q_j x_{ij}-\theta_i \| \leq \Psi(q) \text{ for infinitely many $q \in Q$}.
$$
Clearly $E(1,1,Q,\Psi,\theta) = E(Q,\Psi,\theta)$. As above, we will always assume $\Psi$ is bounded and $\Psi(0)=1$, and these assumptions result in no loss of generality.

For $M > 0$, define 
\begin{align*}
Q(M) &= \cbr{q \in Q : M/2 < |q_j| \leq M \quad \forall 1 \leq j \leq n }, \\
\epsilon(M) &= \min_{q \in Q(M)} \Psi(q).
\end{align*}

\begin{thm}\label{main-thm-2}
Suppose there is a number $a \geq 0$, an increasing function \mbox{$h:(0,\infty) \rightarrow (0,\infty)$}, and an unbounded set $\mathcal{M} \subseteq (0,\infty)$ such that
\begin{align}\label{main-thm-2 e1}
|Q(M)|\epsilon(M)^{a} h(M) \geq M^{a} \quad \forall M \in \mathcal{M}.
\end{align}
Then there is a Borel probability measure $\mu$ supported on $E(m,n,Q,\Psi,\theta)$ such that
\begin{align}\label{main-thm-2 e2}
|\widehat{\mu}(\xi)| \lesssim |\xi|^{-a} \exp\rbr{\frac{\ln |\xi|}{\ln \ln |\xi|}} h(4|\xi|) \quad \forall \xi \in \RR^{mn}, |\xi| > e.
\end{align}
\end{thm}



%
%
Sections \ref{motivation-1} and \ref{motivation-2} discuss motivations for Theorems \ref{main-thm-1} and \ref{main-thm-2}.
Section \ref{applications} contains applications of Theorems \ref{main-thm-1} and \ref{main-thm-2}.
In Section \ref{remarks on proof} we outline the combined proof of Theorems \ref{main-thm-1} and \ref{main-thm-2} and explain its novel aspects.
%
%
The combined proof of Theorems \ref{main-thm-1} and \ref{main-thm-2} constitutes Sections \ref{sec-notation}--\ref{completing the proof}.
Section \ref{appendix} contains the proof of Lemma \ref{eta-upper lemma}. 
We pose questions for further study in Section \ref{Questions for Further Study}.
Section \ref{Acknowledgements} contains acknowledgements.
%



\section{Motivation: Explicit Salem Sets}\label{motivation-1}

The first motivation for our main result is the construction of explicit Salem sets and explicit sets with non-zero Fourier dimension. We start with some definitions and notation.


For $\alpha \geq 0$, the $\alpha$-dimensional Hausdorff content of a set $A \subseteq \RR^d$ is
$$
H^{\alpha}(A) = \inf_{\mathcal{B}} \sum_{B \in \mathcal{B}} (\text{diam} (B))^{\alpha},
$$
where the infimum is over all 
countable collections $\mathcal{B}$ of balls such that $A \subseteq \bigcup_{B \in \mathcal{B}} B$.
%
The Hausdorff dimension of $A$, denoted $\dim_{H}(A)$, is the supremum all of $\alpha \in [0,d]$ such that $H^{\alpha}(A) > 0$. 

If $\mu$ is a finite Borel measure on $\RR^d$, its Fourier transform $\widehat{\mu}$ is defined by
$$
\widehat{\mu}(\xi) = \int_{\RR^d} e^{-2\pi i x \cdot \xi } d\mu(x) \quad \forall \xi \in \RR^d.
$$
If $A \subseteq \RR^d$, the Fourier dimension of $A$, denoted $\dim_{F}(A)$, is the supremum of all $\beta \in [0,d]$ such that 
$$
|\widehat{\mu}(\xi)| \lesssim |\xi|^{-\beta/2} \quad \forall \xi \in \RR^d, \xi \neq 0
$$
for some non-trivial finite Borel measure 
$\mu$ on $\RR^d$ with $\text{supp}(\mu) \subseteq A$. 

As general references for Hausdorff dimension, Fourier dimension, and the Fourier analysis of measures, we give \cite{Mattila}, \cite{Mattila-2}, and \cite{Wolff}. The recent papers \cite{ek-1} and \cite{ek-2} (to name just two) also discuss aspects of the theory of Fourier dimension.

It is well-known (cf. \cite[Chapter 12]{Mattila}, \cite[Chapter 3]{Mattila-2}, \cite[Chapter 8]{Wolff}) that if $A$ is a Borel subset of $\RR^d$, then
\begin{align}
\label{major}
\dim_{H}(A) \geq \dim_{F}(A).
\end{align}
A set $A \subseteq \RR^d$ with $\dim_{H}(A) = \dim_{F}(A)$ is called a Salem set.


Every Borel set in $\RR^d$ of Hausdorff dimension $0$ is a Salem set, $\RR^d$ itself is a Salem set of dimension $d$, and every $(d-1)$-sphere in $\RR^d$ is a Salem set of dimension $d-1$.

Salem \cite{Salem} proved the existence of Salem sets in $\RR$ of arbitrary dimension $\alpha \in (0,1)$ using a random Cantor-type construction. 
Kahane \cite{Kahane-1966} showed that for every $\alpha \in (0,d)$ there is a Salem set in $\RR^d$ of dimension $\alpha$ by considering the images of compact subsets of $[0,1]$ under certain stochastic processes (see also Chapters 17 and 18 of \cite{Kahane}). 
Recently, other random constructions of Salem sets have been given by Bluhm \cite{Bluhm-1}, {\L}aba and Pramanik \cite{LP}, and Shmerkin and Suomala \cite{SS}.
These random constructions do not produce explicit examples of Salem sets; they yield only uncountable families of sets which are almost all Salem sets.


Kaufman \cite{Kaufman} was the first to find an explicit Salem set of dimension $\alpha \notin \cbr{0, d-1,d}$. The set Kaufman proved to be Salem is $E(\ZZ,\Psi_{\tau},0)$, where $\Psi_{\tau}(q) = |q|^{-\tau}$ and $\tau > 1$. An easy and well-known argument (which we give in Section \ref{appendix}) gives
$$
\dim_{H} E(\ZZ,\Psi_{\tau},0) \leq \frac{2}{1+\tau}.
$$
Since $E(\ZZ,\Psi_{\tau},0)$ is a Borel set, \eq{major} implies $\dim_{H} E(\ZZ,\Psi_{\tau},0) \geq \dim_{F} (\ZZ,\Psi_{\tau},0)$. Kaufman showed that for every $\tau > 1$ there is a Borel probability measure $\mu$ with 
support contained in $E(\ZZ,\Psi_{\tau},0) \cap [0,1]$ such that
$$
\widehat{\mu}(\xi) = |\xi|^{-1/(1+\tau)} o(\ln |\xi|) \quad \text{as $|\xi| \rightarrow \infty$},
$$
which implies 
$$
\dim_{F} E(\ZZ,\Psi_{\tau},0) \geq \frac{2}{1+\tau},
$$
and hence that $E(\ZZ,\Psi_{\tau},0)$ is a Salem set. See \cite{Bluhm-2} for a variation of Kaufman's argument with ample details.
In his thesis, Bluhm \cite{Bluhm-thesis} showed that $E(\ZZ,\Psi,0)$ is Salem for any $\Psi$ with $\Psi(q) = \psi(|q|)$ and $\psi: \NN \rightarrow (0,\infty)$ decreasing. Technically, the results of Bluhm and Kaufman are for $E(\NN,\Psi,0)$, not $E(\ZZ,\Psi,0)$, but it is easy to adapt their proofs to $E(\ZZ,\Psi,0)$.


By Dirichlet's approximation theorem, for every $x \in \RR$ there are infinitely many pairs $(p,q) \in \ZZ^2$ for which $|x - p/q| \leq 1/|q|^{2}$. Hence, $E(\ZZ,\Psi_{\tau},0) = \RR$ if $\tau \leq 1$.  A real number $x$ is said to be well approximable if there is a $\tau > 1$ and infinitely many pairs $(p,q) \in \ZZ^2$ for which $|x - p/q| \leq 1/|q|^{1+\tau}$. As the set of well approximable numbers is the union of the sets $E(\ZZ,\Psi_{\tau},0)$ with $\tau > 1$, the result of Kaufman \cite{Kaufman} mentioned above implies the set of well approximable numbers is a Salem set of dimension 1. 

A real number $x$ is said to be badly approximable if there is a positive constant $c(x)$ such that $|x - p/q| > c(x)/|q|^{2}$ for all pairs $(p,q) \in \ZZ^2$. Kaufman \cite{Kaufman-bad} shows, in particular, that the set of badly approximable numbers has positive Fourier dimension. See \cite{QUEFFELEC-RAMARE} and \cite{Jordan-Sahlsten} for extensions of the results of \cite{Kaufman-bad}. It is a classic result of Jarn{\'\i}k \cite{Jarnik-1} that the Hausdorff dimension of the set of badly approximable real numbers is 1. It is unknown whether the set of badly approximable numbers is a Salem set.

If $A \subseteq \RR$ is a set of Fourier dimension $\alpha \in [0,1]$, then it is easy to see the product set $A^d$ has Fourier dimension at least $\alpha$ by considering product measures. A theorem of Gatesoupe \cite{Gatesoupe} implies that if $A \subseteq [0,\infty)$ supports a non-trivial measure and has Fourier dimension $\alpha \in [0,1]$, then $\cbr{x \in \RR^d : |x|_2 \in A}$ has Fourier dimension at least $d-1+\alpha$. Moreover, Gatesoupe's theorem implies that if 
$A \subseteq [0,\infty)$ is a Salem set of dimension $\alpha \in [0,1]$, then $\cbr{x \in \RR^d : |x|_2 \in A}$ is a Salem set in $\RR^d$ of dimension $d-1+\alpha$. 
Combining Gatesoupe's and Kaufman's results yields explicit examples of Salem sets in $\RR^d$ of dimension $\alpha$ for every $\alpha \in [d-1,d]$. Explicit examples of sets (Salem or otherwise) in $\RR^d$ 
with Fourier dimension $\alpha \in (1,d-1)$ were unknown until now.





Theorems \ref{main-thm-1} and \ref{main-thm-2} generalize the theorems of Kaufman \cite{Kaufman} and Bluhm \cite{Bluhm-thesis}. Theorem \ref{main-thm-1} gives many new explicit Salem sets in $\RR$. Some particular Salem sets produced by Theorem \ref{main-thm-1} are discussed in Section \ref{applications}. Theorem \ref{main-thm-2} yields the first examples of explicit set in $\RR^d$ with Fourier dimension $\alpha \in (1,d-1)$. 

\section{Motivation: Metrical Diophantine Approximation}\label{motivation-2}

The second motivation for our main result comes from metrical Diophantine approximation, where there is considerable interest in the Hausdorff dimension of $E(Q,\Psi,\theta)$.



For $\tau \in \RR$, define $\Psi_{\tau}: \ZZ \rightarrow [0,\infty)$ by $\Psi_{\tau}(q) = |q|^{-\tau}$. The classical Jarn{\'\i}k-Besicovitch theorem \cite{Jarnik}, \cite{Bes} is that 
$$
\dim_{H} E(\ZZ,\Psi_{\tau},0) = \min\cbr{\frac{2}{1+\tau},1}.
$$
In the setting of restricted Diophantine approximation, where $Q$ is not necessarily equal to $\ZZ$, Borosh and Fraenkel \cite{BF} showed that 
$$
\dim_{H} E(Q,\Psi_{\tau},0) = \min\cbr{\frac{1+\nu(Q)}{1+\tau},1},
$$
where
$$
\nu(Q) 
= \inf\bcbr{ \nu \geq 0: \sum_{\substack{q \in Q \\ q \neq 0}} |q|^{-\nu} < \infty }. 
$$
Eggleston \cite{Eggleston} previously obtained this result for certain sets $Q$ with 
$\nu(Q)=0$ or $\nu(Q)=1$.

There are also several results for more general functions $\Psi$.
For $\Psi$ of the form $\Psi(q) = \psi(|q|)$ with $\psi: \NN \rightarrow (0,\infty)$ decreasing, Dodson \cite{Dodson} showed that 
$$
\dim_{H} E(\ZZ,\Psi,0) = \min\cbr{\frac{2}{1+\lambda},1},
$$
where
$$
\lambda = \liminf_{M \rightarrow \infty} \frac{-\ln \psi(M)}{\ln M}.
$$
Hinokuma and Shiga \cite{HS} considered the non-monotone function $\Psi_{\tau}^{(HS)}(q) = |\sin q||q|^{-\tau}$ and proved that
$$
\dim_{H} E(\ZZ,\Psi_{\tau}^{(HS)},0) = \min\cbr{\frac{2}{1+\tau},1}.
$$
Dickinson \cite{Dickinson} considered restricted Diophantine approximation with a function $\Psi$ satisfying $\Psi(q) = \psi(|q|)$ with $\psi: \NN \rightarrow (0,\infty)$ and 
$$
\lambda = \liminf_{M \rightarrow \infty} \frac{-\ln \psi(M)}{\ln M} = \limsup_{M \rightarrow \infty} \frac{-\ln \psi(M)}{\ln M}.
$$
Dickinson deduced from the result of Borosh and Fraenkel above that
$$
\dim_{H} E(Q,\Psi,0) = \min\cbr{\frac{1+\nu(Q)}{1+\lambda},1}.
$$
Rynne \cite{Rynne-1998} proved a very general result that implies all of those above. Suppose only that
$\Psi : \ZZ \rightarrow [0,\infty)$ is positive for all $q \in Q$. Let
$$
\eta(Q,\Psi) = \inf\bcbr{\eta \geq 0 : \sum_{\substack{q \in Q \\ q \neq 0}} |q| \rbr{\frac{\Psi(q)}{|q|}}^{\eta} < \infty }.
$$
Rynne showed that
$$
\dim_{H} E(Q,\Psi,0) = \min\cbr{\eta(Q,\Psi),1}.
$$

The main result in the case of inhomogeneous Diophantine approximation (i.e, the case where $\theta$ is non-zero) is due to Levesley \cite{Levesley}. Levesley showed that if $\Psi(q) = \psi(|q|)$ with $\psi: \NN \rightarrow (0,\infty)$ decreasing, and if 
$$
\lambda = \liminf_{M \rightarrow \infty} \frac{-\ln \psi(M)}{\ln M},
$$
then
$$
\dim_{H} E(\ZZ,\Psi,\theta) = \min\cbr{\frac{2}{1+\lambda},1}.
$$ 
By an adaptation of Dickinson's argument from \cite{Dickinson}, the assumption that $\psi$ is decreasing can be replaced by the assumption that
$$
\lambda = \liminf_{M \rightarrow \infty} \frac{-\ln \psi(M)}{\ln M} = \limsup_{M \rightarrow \infty} \frac{-\ln \psi(M)}{\ln M}.
$$

The main content of the formulas above is the lower bounds they give on $\dim_{H} E(Q,\Psi,\theta)$. The $\leq$-half of all the formulas for $\dim_{H} E(Q,\Psi,\theta)$ 
above are implied by the following lemma whose proof is well-known and straightforward. For completeness, we give the proof in Section \ref{appendix}.
\begin{lemma}\label{eta-upper lemma}
\begin{align*}
\dim_{H} E(Q,\Psi,\theta) \leq \min\cbr{\eta(Q,\Psi),1}.
\end{align*}
\end{lemma}



Because of \eq{major}, the Fourier analytic method of Theorem \ref{main-thm-1} stands as an alternative to the usual methods of proving lower bounds on the Hausdorff dimension of $E(Q,\Psi,\theta)$. In fact, Theorem \ref{main-thm-1} implies or implies special cases of all the results for $\dim_{H} E(Q,\Psi,\theta)$ above (details are given in Section \ref{applications}). Moreover, Theorem \ref{main-thm-1} allows us to calculate the Hausdorff dimension of $E(Q,\Psi,\theta)$ in cases that (as far as we know) have not been treated previously in the literature, such as the case where $\theta \neq 0$ and $Q \neq \NN, \ZZ$. 

One particular advantage of the Fourier analytic method of Theorem \ref{main-thm-1} is the ease with which it handles the inhomogeneous case. In the proof of Theorem \ref{main-thm-1} it is trivial to accommodate $\theta \neq 0$, while Levelsey's proof of his result for $\theta \neq 0$ is a non-trivial extension of Dodson's proof for $\theta = 0$.

Our results for $\dim_{H} E(Q,\Psi,\theta)$ are not surprising, and it is likely that they can be obtained by directly extending the methods used by those authors mentioned already in this section, or by applying the powerful and unifying mass transference principle of Beresnevich and Velani \cite{BV}. Of course, these methods cannot be applied to the calculation of the Fourier dimension, which is the main novelty of our paper.

There are analogs of the formulas above for $\dim_{H} E(m,n,Q,\Psi,\theta)$. 
For example Rynne \cite{Rynne-1998} proved
$$
\dim_{H} E(m,n,Q,\Psi,0) = \min\cbr{m(n-1) + \eta(Q,\Psi),mn}
$$
when $\Psi : \ZZ^{n} \rightarrow [0,\infty)$ is positive for all $q \in Q$ and
$$
\eta(Q,\Psi) = \inf\bcbr{\eta \geq 0 : \sum_{\substack{q \in Q \\ q \neq 0}} |q|^m \rbr{\frac{\Psi(q)}{|q|}}^{\eta} < \infty }.
$$
Theorem \ref{main-thm-2} (via \eq{major}) provides a lower bound on $\dim_{H} E(m,n,Q,\Psi,\theta)$, but it does not reach the true value of $\dim_{H} E(m,n,Q,\Psi,\theta)$ for any known case with $mn > 1$.

\section{Applications}\label{applications}

In this section we will present several consequences of Theorem \ref{main-thm-1} that give 
new 
families of explicit Salem sets and imply formulas for $\dim_H E(Q,\Psi,\theta)$ discussed in Section \ref{motivation-2}. We will also present a typical consequence of Theorem \ref{main-thm-2} that yields explicit  
sets in $\RR^d$ with Fourier dimension strictly between $1$ and $d-1$.

\begin{thm}\label{prop-1}
Assume $\Psi$ is of the form $\Psi(q) = \psi(|q|)$ with $\psi: \NN \rightarrow (0,\infty)$ a decreasing function. Assume there is an increasing function $h:(0,\infty) \rightarrow (0,\infty)$ such that
\begin{align}\label{prop-1 1}
|Q(M)| \geq M / h(M) \quad \forall M \in \mathbb{N}
\end{align}
and 
$$
\lim_{x \rightarrow \infty} \frac{\ln h(x)}{\ln x} = 0
$$
Then $E(Q,\Psi,\theta)$ is a Salem set of dimension $\min\cbr{2/(1+\lambda),1}$, where
\begin{align}\label{prop-1 3}
\lambda = \liminf_{M \rightarrow \infty} \frac{-\ln(\psi(M))}{\ln M}.
\end{align}
\end{thm}
\begin{proof}
Since $\Psi$ is bounded, $\lambda \geq 0$.

Let $\lambda^{\prime} < \lambda$ and $\delta > 0$. By \eq{prop-1 3}, $\psi(M) < M^{-\lambda^{\prime}}$ for all large $M \in \mathbb{N}$. It follows that
$$
\sum_{\substack{q \in Q \\ q \neq 0}} |q|\rbr{\frac{\Psi(q)}{|q|}}^{(2+\delta)/(1+\lambda^{\prime})} 
\lesssim 
\sum_{\substack{q \in Q \\ q \neq 0}} |q|^{-(1+\delta)} < \infty.
$$
Therefore, by applying Lemma \ref{eta-upper lemma} and then letting $\lambda^{\prime} \rightarrow \lambda$ and $\delta \rightarrow 0$, we have
$$
\dim_{H} E(Q,\Psi,\theta) \leq \min\cbr{\frac{2}{1+\lambda},1}.
$$
If $\lambda = \infty$, this argument shows $\dim_{H} E(Q,\Psi,\theta) = \dim_{F} E(Q,\Psi,\theta) = 0$.

Assume $0 \leq \lambda < \infty$. Let $\lambda^{\prime} > \lambda$. By \eq{prop-1 3}, there is an infinite set $\mathcal{M} \subseteq \NN$ such that $\psi(M) > M^{-\lambda^{\prime}}$ for all $M \in \mathcal{M}$. Since $\psi$ is decreasing, it follows that $\epsilon(M) > M^{-\lambda^{\prime}}$ for all $M \in \mathcal{M}$. Combining this with \eq{prop-1 1}, we see that \eq{main-thm-1 e1} holds with $a = 1/(1+\lambda^{\prime})$. Since $\lambda^{\prime} > \lambda$ is arbitrary, Theorem \ref{main-thm-1 e1} gives
$$
\dim_{F} E(Q,\Psi,\theta) \geq \min\cbr{\frac{2}{1+\lambda},1}.
$$
\end{proof}
Theorem \ref{prop-1} implies the result of Dodson \cite{Dodson} for $\dim_{H} E(\ZZ,\Psi,0)$ 
discussed in Section \ref{motivation-2}. Theorem \ref{prop-1} also implies the formula for $\dim_{H} E(\ZZ,\Psi,\theta)$ due to Levesley \cite{Levesley} mentioned in Section \ref{motivation-2}.

\begin{thm}\label{cor-8}
Assume $\Psi$ is of the form $\Psi(q) = \psi(|q|)$ with $\psi: \NN \rightarrow (0,\infty)$. Assume
\begin{align}\label{cor-8 1}
\lambda = \liminf_{M \rightarrow \infty} \frac{-\ln \psi(M)}{\ln M} = \limsup_{M \rightarrow \infty} \frac{-\ln \psi(M)}{\ln M}
\end{align}
and
\begin{align}\label{cor-8 2}
\sum_{\substack{q \in Q \\ q \neq 0}} |q|^{-1} = \infty.
\end{align}
Then $E(Q,\Psi,\theta)$ is a Salem set of dimension $\min\cbr{2/(1+\lambda),1}$.
\end{thm}
\begin{proof}
Since $\Psi$ is bounded, $\lambda \geq 0$. 

The same argument as in the proof of Theorem \ref{prop-1} shows 
$$
\dim_{H} E(Q,\Psi,\theta) \leq \min\cbr{\frac{2}{1+\lambda},1}.
$$
If $\lambda = \infty$, the argument shows $\dim_{H} E(Q,\Psi,\theta) = \dim_{F} E(Q,\Psi,\theta) = 0$.

Assume $0 \leq \lambda < \infty$. Seeking a contradiction suppose, $|Q(M)| < M/\ln^2(M)$ for all large $M \in \NN$. Then
\begin{align*}
\sum_{\substack{q \in Q \\ q \neq 0}} |q|^{-1}
&=
\sum_{k=0}^{\infty} \sum_{q \in Q(2^{k})} |q|^{-1} 
\lesssim
\sum_{k=0}^{\infty} 2^{-k} \sum_{q \in Q(2^{k})} 1 \\
&\lesssim
\sum_{k=0}^{\infty} \frac{1}{\ln^2(2^k)} =
\frac{1}{\ln^2(2)} \sum_{k=0}^{\infty} \frac{1}{k^2} < \infty,
\end{align*}
which contradicts \eq{cor-8 2}. So there is an infinite set $\mathcal{M} \subseteq \mathbb{N}$ such that 
$$
|Q(M)| \geq M/\ln^2(M) \quad \forall M \in \mathcal{M}.
$$ 
Let $\lambda^{\prime} > \lambda$ be given. By \eq{cor-8 1}, $\psi(M) > M^{-\lambda^{\prime}}$ for all large $M \in \NN$. Therefore $\epsilon(M) > M^{-\lambda^{\prime}}$ for all large $M \in \NN$. After removing finitely elements of $\mathcal{M}$, we have 
$$
\epsilon(M) > M^{-\lambda^{\prime}} \quad \forall M \in \mathcal{M}.
$$ 
Then \eq{main-thm-1 e1} holds with $a = 1/(1+\lambda^{\prime})$ and $h(x) = \ln^2(x)$. Since $\lambda^{\prime} > \lambda$ is arbitrary, Theorem \ref{main-thm-1} gives
$$
\dim_{F} E(Q,\Psi,\theta) \geq \min\cbr{\frac{2}{1+\lambda},1}.
$$
\end{proof}
Theorem \ref{cor-8} implies the result of Dickinson \cite{Dickinson} for $\dim_{H} E(Q,\Psi,0)$ discussed in Section \ref{motivation-2} in the case $\nu(Q) = 1$. 
Consequently, it also implies the results of Borosh and Fraenkel \cite{BF} and Eggleston \cite{Eggleston} in the case $\nu(Q) = 1$.
Theorem \ref{cor-8} implies the variation of the result of Levesley \cite{Levesley} for $\dim_{H} E(\ZZ,\Psi,\theta)$ mentioned in Section \ref{motivation-2} that uses Dickinson's argument from \cite{Dickinson}.

As far as we know, Theorems \ref{prop-1} and \ref{cor-8} represent the first calculation of $\dim_{H} E(Q,\Psi,\theta)$ in the case where $\theta \neq 0$ and $Q \neq \NN, \ZZ$.

\begin{thm}\label{cor-10}
Suppose $\Psi_{\tau}^{(HS)}: \ZZ \rightarrow (0,\infty)$ is defined by $\Psi_{\tau}^{(HS)}(q) = |\sin q| |q|^{-\tau}$. Then $E(\ZZ,\Psi_{\tau}^{(HS)},\theta)$ is a Salem set of dimension of $\min\cbr{2/(1+\tau),1}$.
\end{thm}
\begin{proof}
For every $\epsilon > 0$, 
\begin{align*}
\sum_{\substack{q \in \ZZ \\ q \neq 0}} |q| \rbr{\frac{\Psi(q)}{|q|}}^{2/(1+\tau) + \epsilon}
=
\sum_{\substack{q \in \ZZ \\ q \neq 0}} |q|^{-1-(1+\tau)\epsilon} |\sin q|^{2/(1+\tau) + \epsilon} 
< \infty.
\end{align*}
So, by Lemma \ref{eta-upper lemma},
$$
\dim_{H} E(\ZZ,\Psi_{\tau}^{(HS)},\theta) \leq \min\cbr{\frac{2}{1+\tau},1}.
$$
Let $Q = \cbr{q \in \NN : |\sin q| \geq 1/2}$. Clearly
$$
\min_{q \in Q(M)} \Psi_{\tau}^{(HS)}(q) \geq \frac{1}{2}M^{-\tau} \quad \forall M \in \mathbb{N}.
$$
It is also easy to see that 
$$
|Q(M)| \geq \frac{M}{4\pi}
$$
for all large $M$ (for instance, by noting $Q$ contains the nearest integer(s) to $(2k+1)\pi/2$ for every $k \in \NN$). 
It follows that \eq{main-thm-1 e1} holds with $a = 1/(1+\tau)$ and $h(x)=4\pi$. Therefore Theorem \ref{main-thm-1} and the fact $E(Q,\Psi_{\tau}^{(HS)},\theta) \subseteq E(\ZZ,\Psi_{\tau}^{(HS)},\theta)$ implies
$$
\dim_{F} E(\ZZ,\Psi_{\tau}^{(HS)},\theta) \geq \min\cbr{\frac{2}{1+\tau},1}.
$$
\end{proof}
Theorem \ref{cor-10} implies the formula for $\dim_{H} E(\ZZ,\Psi_{\tau}^{(HS)},0)$ due to Hinokuma and Shiga \cite{HS} mentioned in Section \ref{motivation-2}.




Finally, we give a typical consequence of Theorem \ref{main-thm-2} that yields (in particular) explicit 
sets in $\RR^d$ with Fourier dimension strictly between $1$ and $d-1$. 
The Hausdorff dimension of the sets is also determined for comparison.

\begin{thm}\label{mn-app}
Assume $\Psi: \ZZ^n \rightarrow [0,\infty)$ is of the form $\Psi(q) = \psi(|q|)$ with $\psi: \NN \rightarrow (0,\infty)$ and
\begin{align}\label{mn-app e1}
\lambda = \liminf_{M \rightarrow \infty} \frac{-\ln \psi(M)}{\ln M} = \limsup_{M \rightarrow \infty} \frac{-\ln \psi(M)}{\ln M}.
\end{align}
Assume $h:(0,\infty) \rightarrow (0,\infty)$ is an increasing function such that
\begin{align}\label{mn-app e2}
\lim_{x \rightarrow \infty} \frac{\ln h(x)}{\ln x} = 0.
\end{align}
Assume there is an unbounded set $\mathcal{M} \subseteq (0,\infty)$ such that
\begin{align}\label{mn-app e3}
|Q(M)|h(M) \geq M^{n} \quad \forall M \in \mathcal{M}.
\end{align}
Then 
$$
\dim_{H} E(m,n,Q,\Psi,0) = \min\cbr{m(n-1) + \frac{m+n}{1+\lambda},mn}
$$
and
$$
\dim_{F} E(m,n,Q,\Psi,0) \geq \min\cbr{\frac{2n}{1+\lambda},mn}.
$$
\end{thm}
\begin{proof}
Since $\Psi$ is bounded, $\lambda \geq 0$.

Let $\lambda^{\prime} < \lambda$ and $\delta > 0$. By \eq{mn-app e1}, $\psi(M) < M^{-\lambda^{\prime}}$ for all large $M \in \mathbb{N}$. It follows that
$$
\sum_{\substack{q \in Q \\ q \neq 0}} |q|^m \rbr{\frac{\Psi(q)}{|q|}}^{(m+n+\delta)/(1+\lambda^{\prime})} 
\lesssim 
\sum_{\substack{q \in Q \\ q \neq 0}} |q|^{-(n+\delta)} < \infty.
$$
Since $\lambda^{\prime} < \lambda$ and $\delta > 0$ are arbitrary, we have 
$$
\sum_{\substack{q \in Q \\ q \neq 0}} |q|\rbr{\frac{\Psi(q)}{|q|}}^{\eta} < \infty \quad \forall \eta > \frac{m+n}{1+\lambda}.
$$
Let $\lambda^{\prime} > \lambda$ and $\delta > 0$. By \eq{mn-app e1}, $\psi(M) > M^{-\lambda^{\prime}}$ for all large $M \in \mathbb{N}$. Choose a sequence $(M_k)_{k \in \NN}$ of numbers in $\mathcal{M}$ such that $M_k \geq 2M_{k-1}$. Then
\begin{align*}
\sum_{\substack{q \in Q \\ q \neq 0}} |q|^m \rbr{\frac{\Psi(q)}{|q|}}^{(m+n-\delta)/(1+\lambda^{\prime})}
&\gtrsim 
\sum_{\substack{q \in Q \\ q \neq 0}} |q|^{-(n-\delta)}
\gtrsim
\sum_{k \in \NN} \sum_{q \in Q(M_k)} |q|^{-(n-\delta)} \\
&\gtrsim
\sum_{k \in \NN} M_{k}^{-(n-\delta)} |Q(M_k)|
\gtrsim
\sum_{k \in \NN} M_{k}^{\delta}(h(M_k))^{-1} \\
&\gtrsim
\sum_{k \in \NN} M_{k}^{\delta/2} = \infty
\end{align*}
Since $\lambda^{\prime} > \lambda$ and $\delta > 0$ are arbitrary, we have 
$$
\sum_{\substack{q \in Q \\ q \neq 0}} |q|\rbr{\frac{\Psi(q)}{|q|}}^{\eta} = \infty \quad \forall \eta < \frac{m+n}{1+\lambda}.
$$
By the result of Rynne \cite{Rynne-1998} for $\dim_{H} E(m,n,Q,\Psi,0)$ discussed in Section \ref{motivation-2}, we have
$$
\dim_{H} E(m,n,Q,\Psi,0) = \min \cbr{m(n-1) + \frac{m+n}{1+\lambda},mn}.
$$
If $\lambda = \infty$, this argument shows $\dim_{H} E(m,n,Q,\Psi,0) = m(n-1)$.

Assume $0 \leq \lambda < \infty$. Let $\lambda^{\prime} > \lambda$. By \eq{mn-app e1}, $\psi(M) > M^{-\lambda^{\prime}}$ for all large $M \in \NN$. Therefore $\epsilon(M) > M^{-\lambda^{\prime}}$ for all large $M \in \NN$. After removing finitely many elements of $\mathcal{M}$, we have $\epsilon(M) > M^{-\lambda^{\prime}}$ for all $M \in \mathcal{M}$. Combining this with \eq{mn-app e3}, we see that \eq{main-thm-2 e1} holds with $a = n/(1+\lambda^{\prime})$. 
As $\lambda^{\prime} > \lambda$ is arbitrary, Theorem \ref{main-thm-2} implies 
$$
\dim_{F} E(m,n,Q,\Psi,0) \geq \min\cbr{\frac{2n}{1+\lambda},mn}.
$$
If $\lambda = \infty$,  the desired lower bound is $\dim_{F} E(m,n,Q,\Psi,0) \geq 0$, which holds by definition.
\end{proof}

With the additional assumption $\frac{m+n}{m-1} < \lambda + 1 < 2n$, Theorem \ref{mn-app} implies that $E(m,n,Q,\Psi,0)$ is a subset of $\RR^{mn}$ with Fourier dimension strictly between $1$ and $mn-1$. 
As a concrete example, if $m=4$, $n=2$, $\lambda = 2$, $Q = \ZZ^n$, and $\Psi(q) = |q|^{-\lambda}$, then Theorem \ref{mn-app} implies 
$$
1 < \frac{4}{3} = \frac{2n}{1+\lambda} \leq \dim_{F} E \leq \dim_{H} E = m(n-1) + \frac{m+n}{1+\lambda} = 6 < 7 = mn-1,
$$
where we have put $E = E(m,n,Q,\Psi,0)$ for brevity.


\section{Remarks on the proof of Theorems \ref{main-thm-1} and \ref{main-thm-2}}\label{remarks on proof}

Since Theorem \ref{main-thm-2} is a generalization of Theorem \ref{main-thm-1}, we will give a single unified proof. The proof is in Sections \ref{sec-notation}--\ref{completing the proof}. In this section, we will outline the proof and explain its novel aspects. 

The proof of Theorems \ref{main-thm-1} and \ref{main-thm-2} is essentially a generalization of the proofs of the theorems of Kaufman \cite{Kaufman} and Bluhm \cite{Bluhm-thesis} mentioned in Section \ref{motivation-1}. The reformulations of Kaufman's proof by Bluhm \cite{Bluhm-2} and Wolff \cite{Wolff} were also valuable guides.

In order to explain the novel aspects of the proof of Theorems \ref{main-thm-1} and \ref{main-thm-2}, we will begin with an outline of Kaufman's proof and then gradually generalize it as we build towards the proof of Theorem \ref{main-thm-1} and Theorem \ref{main-thm-2}.

All the proofs have the same general form. The measure $\mu$ is defined as the weak limit of a sequence absolutely continuous measures $(\mu_k)_{k=0}^{\infty}$. The density of $\mu_k$ is the product of functions $\chi_0 F_{M_1} \cdots F_{M_{k}}$. Here $\chi_0$ is a bump function intended to restrict the support of the measures to a common compact set, and $(M_k)_{k=1}^{\infty}$ is a sequence of positive real numbers (whose precise definition will not be discussed in this outline). 
The functions $F_M$ are designed to have two important properties. The first property is that the support of $F_M$ is such that the infinite product $\prod_{i=1}^{\infty}F_{M_i}$, and hence $\mu$, is supported on the appropriate version of $E(Q,\Psi,\theta)$ (or $E(m,n,Q,\Psi,\theta)$). The second property of $F_M$ is a Fourier decay estimate. The desired Fourier decay estimate on $\mu$ is ultimately deduced from this Fourier decay estimate on $F_{M}$. The functions $F_M$ are the key to the proof, so our outline will focus on them.

Kaufman \cite{Kaufman} constructed a measure $\mu$ on $E(\ZZ,\Psi_{\tau},0)$, where $\Psi_{\tau}(q) = |q|^{-\tau}$ and $\tau > 1$, with 
\begin{align*}
|\widehat{\mu}(\xi)| = |\xi|^{-1/(1+\tau)} o(\ln |\xi|) \quad \text{as $|\xi| \rightarrow \infty$}.
\end{align*}
We will outline a slightly simplified version of Kaufman's proof that gives a slightly slower Fourier decay estimate. The proof of Theorem \ref{main-thm-1} and Theorem \ref{main-thm-2} is closer to this simplified version than it is to Kaufman's original proof. 
%
%
Define 
$$
F_{M}(x) = \frac{1}{|\mathcal{P}(M)|} \sum_{q \in \mathcal{P}(M)} \sum_{k \in \ZZ} \epsilon(M)^{-1}\phi(\epsilon(M)^{-1}(xq  - k))
\quad \forall x \in \RR.
$$
Here $\mathcal{P}$ is the set of prime numbers, $\mathcal{P}(M) = \cbr{q \in \mathcal{P} : M/2 < |q| \leq M }$, $\epsilon(M) = \Psi_{\tau}(M) = M^{-\tau}$, and $\phi: \RR \rightarrow \RR$ is an arbitrary $C^K$ function with support contained in $[-1,1]$ and $K$ sufficiently large. 
Since $\epsilon(M) \leq \Psi_{\tau}(q)$ for all $q \in \mathcal{P}(M)$, the support of $F_M$ is contained in
$$
\cbr{x \in \RR : \| qx \| \leq \Psi_{\tau}(q) \text{ for some $q \in \mathcal{P}(M)$}}.
$$
Consequently, if $(M_k)_{k=1}^{\infty}$ grows quickly enough, the support of $\mu$ is contained in $E(\mathcal{P},\Psi_{\tau},0)$, which is a subset of $E(\ZZ,\Psi_{\tau},0)$. We now describe the key Fourier decay estimate on $F_{M}$. Basic properties of the Fourier transform yield
$$
|\widehat{F_{M}}(\ell)| \lesssim \dfrac{|\mathcal{P}(M) \cap D(\ell)|}{|\mathcal{P}(M)|} (1 + \epsilon(M) M^{-1} |\ell|)^{-K} \quad \forall \ell \in \ZZ,
$$
where $D(\ell)$ is the set of integers which divide $\ell$. 
We estimate each factor on the right-hand side separately.
First, we require $K \geq \frac{1}{1+\tau}$ so that
$$
(1 + \epsilon(M) M^{-1} |\ell|)^{-K} \leq (\epsilon(M) M^{-1} |\ell|)^{-1/(1+\tau)} = M |\ell|^{-1/(1+\tau)}.
$$
Next, by the fundamental theorem of arithmetic, 
$$
|\mathcal{P}(M) \cap D(\ell)| \leq 2 \frac{\ln |l|}{\ln M} \quad \forall \ell \in \ZZ, \ell \neq 0, M \geq 2.
$$
Finally, by the density of the primes, 
$$
|\mathcal{P}(M)| \gtrsim \frac{M}{\ln M} \quad \forall M \geq 2.
$$
Putting it all together, we obtain 
$$
|\widehat{F_{M}}(\ell)| \lesssim |\ell|^{-1/(1+\tau)} \ln|\ell| \quad \forall \ell \in \ZZ, \ell \neq 0, M \geq 2.
$$
If $K > 1 + a$, we can use this to deduce (for instance)
$$
|\widehat{\mu}(|\xi|)| \lesssim |\xi|^{-1/(1+\tau)} \delta(|\xi|) \ln^{1+\delta} |\xi|  \quad \forall \xi \in \RR, |\xi| > e,
$$
for any prescribed $\delta > 0$.

Our next step will be generalizing Kaufman's argument to $E(Q,\Psi_{\tau},0)$, where $Q$ is any infinite subset of $\ZZ$. 
We now take
$$
F_{M}(x) = \frac{1}{|Q(M)|} \sum_{q \in Q(M)} \sum_{k \in \ZZ} \epsilon(M)^{-1}\phi(\epsilon(M)^{-1}(xq  - k))
\quad \forall x \in \RR,
$$ 
where $Q(M) = \cbr{q \in Q : M/2 < |q| \leq M }$, $\epsilon(M) = \Psi_{\tau}(M) = M^{-\tau}$, and $\phi: \RR \rightarrow \RR$ is a $C^K$ function with support contained in $[-1,1]$ and $K$ sufficiently large. 
As before, since $\epsilon(M) \leq \Psi_{\tau}(q)$ for all $q \in Q(M)$, the support of $F_M$ is contained in
$$
\cbr{x \in \RR : \| qx \| \leq \Psi_{\tau}(q) \text{ for some $q \in Q(M)$}},
$$
and therefore the support of $\mu$ is contained in $E(Q,\Psi_{\tau},0)$, provided $(M_k)_{k=1}^{\infty}$ grows sufficiently quickly. 
The Fourier decay estimate on $F_{M}$ is different from the one in Kaufman's proof. It starts the same way, with the bound
\begin{align}\label{start decay outline}
|\widehat{F_{M}}(\ell)| \lesssim \dfrac{|Q(M) \cap D(\ell)|}{|Q(M)|} (1 + \epsilon(M) M^{-1} |\ell|)^{-K} \quad \forall \ell \in \ZZ,
\end{align}
where $D(\ell)$ is the set of integers which divide $\ell$.
We estimate each factor on the right-hand side separately. 
First, we require $K \geq a$ so that
\begin{align*}
(1 + \epsilon(M) M^{-1} |\ell|)^{-K} \leq (\epsilon(M) M^{-1} |\ell|)^{-a}
\end{align*}
Since $Q$ is not required to have any specific arithmetic structure, we cannot estimate $|Q(M) \cap D(\ell)|$ as simply as in Kaufman's argument. 
Instead we bound $|Q(M) \cap D(\ell)|$ using the divisor bound of Wigert \cite{Wigert} to obtain that for every $\zeta > \ln 2$ there is an $L_{\zeta} \in \NN$ such that
$$
|Q(M) \cap D(\ell)| \leq \exp\rbr{\frac{\zeta \ln|\ell|}{\ln \ln |\ell|}} \quad \forall \ell \in \ZZ, |\ell| \geq L_{\zeta}.
$$
Finally, we need a lower bound on $|Q(M)|$. Since $Q$ is an arbitrary infinite set of integers, we cannot say much in general. But we know there must exist a number $a \geq 0$, an increasing function \mbox{$h:(0,\infty) \rightarrow (0,\infty)$}, and an unbounded set $\mathcal{M} \subseteq (0,\infty)$ such that 
\begin{align}\label{main-thm-1 e1 again}
|Q(M)|\epsilon(M)^{a} h(M) \geq M^{a} \quad \forall M \in \mathcal{M}.
\end{align}
Though it is not necessary for the proof, we can always choose $h$ so that $\lim_{x \rightarrow \infty} \ln h(x) / ln x = 0$. 
Putting it all together, we obtain
$$
|\widehat{F_{M}}(\ell)| \lesssim |\ell|^{-a} \exp\rbr{ \frac{\zeta \ln |\ell|}{\ln \ln |\ell|} } h(M) \quad \forall \ell \in \ZZ, |\ell| \geq L_{\zeta}, M \in \mathcal{M}.
$$
From this, provided $K > 1 + a$, it can be deduced that
$$
|\widehat{\mu}(|\xi|)| \lesssim |\xi|^{-a} \exp\rbr{ \frac{ \ln |\xi|}{\ln \ln |\xi|} } h(4 |\xi|) \quad \forall \xi \in \RR, |\xi| > e.
$$

We discuss this result briefly before moving on to the next generalization. Note that if $\lim_{x \rightarrow \infty} \ln h(x) / ln x = 0$ (which we can always achieve), then $\exp\rbr{ \frac{ \ln |\xi|}{\ln \ln |\xi|} } h(4|\xi|)$ goes to $\infty$ as $|\xi| \rightarrow \infty$ slower than any power of $|\xi|$, just like $\ln|\xi|$. So having the factor $\exp(\ln|\xi| / \ln \ln |\xi|) h(4|\xi|)$ rather than $\ln|\xi|$ does not cost us anything in terms of Fourier dimension. However, sparsity of $Q$ will decrease the exponent $a$ and (therefore) the Fourier dimension lower bound. Consider the following two examples. 
First suppose $Q$ is the set of primes shifted up by $1$, i.e., $Q = \cbr{p+1 : p \in \mathcal{P}}$. Unlike the set of primes $\mathcal{P}$, the shifted set $Q$ has no obviously useful arithmetic structure. However, $\mathcal{P}$ and $Q$ have essentially the same density: $|\mathcal{P}| \approx |Q(M)| \approx M/\log M$ for all $M$ large enough. 
In fact, \eq{main-thm-1 e1 again} holds with $a = 1/(1+\tau)$, $h(x) = 4\ln(x+1)$, and $\mathcal{M} = [3,\infty)$, so that $|\widehat{\mu}(|\xi|)| \lesssim |\xi|^{-1/(1+\tau)} \exp\rbr{ \frac{ \ln |\xi|}{\ln \ln |\xi|} } \ln(4|\xi|+1)$ for $|\xi| > e$, and therefore $\dim_F E(Q,\Psi_{\tau},0) \geq 2/(1+\tau)$. 
For comparison, Kaufman's argument applied to $E(\mathcal{P},\Psi_{\tau},0)$ leads to $|\widehat{\mu}(|\xi|)| \lesssim |\xi|^{-1/(1+\tau)} \ln|\xi|$ for $|\xi| > e$ and thus the same lower bound $\dim_F E(\mathcal{P},\Psi_{\tau},0) \geq 2/(1+\tau)$. For the second example, suppose $Q$ is the set of perfect squares, i.e., $Q = \cbr{n^2 : n \in \NN}$, which is much sparser than $\mathcal{P}$. Indeed $|Q(M)| \approx M^{1/2}$ for all $M$ large enough, and \eq{main-thm-1 e1 again} holds with with $a = 1/2(1+\tau)$, $h(x) = 10$, and $\mathcal{M} = [9,\infty)$. So we get  $|\widehat{\mu}(|\xi|)| \lesssim |\xi|^{-1/2(1+\tau)} \exp\rbr{ \frac{ \ln |\xi|}{\ln \ln |\xi|} }$ for $|\xi| > e$, and therefore $\dim_F E(Q,\Psi_{\tau},0) \geq 1/(1+\tau)$.


Recall that Bluhm \cite{Bluhm-thesis} extended Kaufman's result to $E(\ZZ,\Psi,0)$ with $\Psi(q) = \psi(|q|)$ and $\psi: \NN \rightarrow (0,\infty)$ decreasing. The next step in our outline is to consider $E(Q,\Psi,0)$, where $Q$ is any infinite subset of $\ZZ$, and $\Psi$ is any function mapping $\ZZ \rightarrow [0,\infty)$ that is positive on $Q$. 
In fact, after replacing all instances of $\Psi_{\tau}$ by $\Psi$, the preceding argument goes through almost word for word. We just need to modify the definition $\epsilon(M)$. The only important feature of $\epsilon(M)$ in the preceding argument is that it is a positive number satisfying $\epsilon(M) \leq \Psi_{\tau}(q)$ for all $q \in Q(M)$. So we simply replace $\epsilon(M) =\Psi_{\tau}(M) = M^{-\tau}$ by $\epsilon(M) = \min_{q \in Q(M)}\Psi(q)$.


The next step in our outline is to generalize to $E(Q,\Psi,\theta)$ with $\theta$ being any real number. We again need only a very minor modification in the argument. In the definition of $F_M$, we replace 
$xq  - k$ by $xq  - k - \theta$. So we now take
$$
F_{M}(x) = \frac{1}{|Q(M)|} \sum_{q \in Q(M)} \sum_{k \in \ZZ} \epsilon(M)^{-1}\phi(\epsilon(M)^{-1}(xq  - k - \theta))
\quad \forall x \in \RR.
$$ 
Then the support of $F_M$ is contained in
$$
\cbr{x \in \RR : \| qx - \theta \| \leq \Psi(q) \text{ for some $q \in Q(M)$}},
$$
and therefore the support of $\mu$ is contained in $E(Q,\Psi,\theta)$. The replacement of $xq  - k$ by $xq  - k - \theta$ leaves the estimate \eq{start decay outline} unchanged. (This may not be easy to see here, but it is easy to see when one reads the details of the proof of Theorem \ref{main-thm-1} and Theorem \ref{main-thm-2} in Section \ref{key function}). Therefore the rest of argument proceeds exactly as above.

The final step in our outline is extending the argument to $E(m,n,Q,\Psi,\theta)$, where $m, n \in \NN$, $Q$ is an infinite subset of $\ZZ^n$, $\Psi: \ZZ^n \rightarrow [0,\infty)$ is positive on $Q$, and $\theta \in \RR^m$. To define the functions $F_M$, we need a few preliminaries. Define 
\begin{align*}
Q(M) &= \cbr{q \in Q : M/2 < |q_j| \leq M \quad \forall 1 \leq j \leq n }, \\
\epsilon(M) &= \min_{q \in Q(M)} \Psi(q),
\end{align*}
and let $\phi: \RR \rightarrow \RR$ be any $C^K$ function with support contained in $[-1,1]^m$ and $K$ sufficiently large. 
For $x = (x_{11}, \ldots, x_{1n}, \ldots, x_{m1}, \ldots, x_{mn}) \in \RR^{mn}$ and $q = (q_1,\ldots,q_n) \in \ZZ^n$, define the product $xq$ by identifying $x$ with the $m \times n$ matrix whose $ij$-entry is $x_{ij}$.
Finally, define
$$
F_{M}(x) = \frac{1}{|Q(M)|} \sum_{q \in Q(M)} \sum_{k \in \ZZ^{m}} \epsilon(M)^{-1}\phi(\epsilon(M)^{-1}(xq  - k - \theta))
\quad \forall x \in \RR^{mn}.
$$
Since $\epsilon(M) \leq \Psi(q)$ for all $q \in Q(M)$, the support of $F_M$ is contained in
$$
\cbr{x \in \RR : \max_{1 \leq i \leq m} \| \sum_{j=1}^{n} q_{j} x_{ij} - \theta_{i} \| \leq \Psi(q) \text{ for some $q \in Q(M)$}},
$$
and therefore the support of $\mu$ is contained in $E(m,n,Q,\Psi,\theta)$, provided $(M_k)_{k=1}^{\infty}$ grows sufficiently quickly. As before, the key Fourier decay estimate on $F_{M}$ begins with the relatively straightforward bound
\begin{align*}
|\widehat{F_{M}}(\ell)| \lesssim \dfrac{|Q(M) \cap D(\ell)|}{|Q(M)|} (1 + \epsilon(M) M^{-1} |\ell|)^{-K} \quad \forall \ell \in \ZZ^{mn}.
\end{align*}
Of course, $D(\ell)$ here is no longer the set of integers dividing $\ell$. Now $D(\ell)$ is the set of points in $\ZZ^{n}$ obeying a certain more complicated arithmetic relationship with the point $\ell \in \ZZ^{mn}$. 
However, we still use Wigert's divisor bound to show that for every $\zeta > \ln 2$ there is an $L_{\zeta} \in \NN$ such that
\begin{align*}
|Q(M) \cap D(\ell)| \leq \exp\rbr{\frac{\zeta \ln|\ell|}{\ln \ln |\ell|}} \quad \forall \ell \in \ZZ^{mn}, |\ell| \geq L_{\zeta}.
\end{align*}
For the lower bound on $|Q(M)|$, we still know there must exist a number $a \geq 0$, an increasing function \mbox{$h:(0,\infty) \rightarrow (0,\infty)$}, and an unbounded set $\mathcal{M} \subseteq (0,\infty)$ such that 
\begin{align*}
|Q(M)|\epsilon(M)^{a} h(M) \geq M^{a} \quad \forall M \in \mathcal{M}.
\end{align*}
As before, we can always choose $h$ so that $\lim_{x \rightarrow \infty} \ln h(x) / ln x = 0$, but it is not necessary for the proof.
Finally, we are still permitted to require $K \geq a$ so that
\begin{align*}
(1 + \epsilon(M) M^{-1} |\ell|)^{-K} \leq (\epsilon(M) M^{-1} |\ell|)^{-a}.
\end{align*}
Thus we obtain
$$
|\widehat{F_{M}}(\ell)| \lesssim |\ell|^{-a} \exp\rbr{ \frac{\zeta \ln |\ell|}{\ln \ln |\ell|} } h(M) \quad \forall \ell \in \ZZ^{mn}, |\ell| \geq L_{\zeta}, M \in \mathcal{M}.
$$
If $K > mn + a$, we can then show
$$
|\widehat{\mu}(|\xi|)| \lesssim |\xi|^{-a} \exp\rbr{ \frac{ \ln |\xi|}{\ln \ln |\xi|} } h(4|\xi|) \quad \forall \xi \in \RR^{mn}, |\xi| > e.
$$

The proof of Theorems \ref{main-thm-1} and \ref{main-thm-2} constitutes Sections \ref{sec-notation}--\ref{completing the proof}. We conclude the current section by describing the contents of Sections \ref{sec-notation}--\ref{completing the proof}, so that the reader can easily find the details of the steps from the outline above. 

Section \ref{sec-notation} preemptively clarifies some potentially confusing notation for the Fourier transform. The function $\phi$ and the associated parameter $K$ are introduced in Section \ref{sec-phi-Phi}. Additionally, Section \ref{sec-phi-Phi} defines the function $\Phi^{\epsilon}_{q,\theta}$ and works out its Fourier transform. The purpose of defining the function $\Phi^{\epsilon}_{q,\theta}$ is to make it easier to establish certain properties of $F_M$. The precise definition of $D(\ell)$ for $\ell \in \ZZ^{mn}$ is also given in Section \ref{sec-phi-Phi}. In Section \ref{key function}, $F_M$ is defined in terms $\Phi^{\epsilon}_{q,\theta}$, some simple properties of $\widehat{F_M}$ are worked out using $\widehat{\Phi^{\epsilon}_{q,\theta}}$, and the support of $F_M$ is described. In Section \ref{fourier decay FM}, the key Fourier decay property of $F_M$ is established. In the course of doing so, the statement of Wigert's divisor bound and the details of how it is used to bound $|Q(M) \cap D(\ell)|$ are given. 
Section \ref{main lemma} contains the statement and proof of an important lemma. The lemma is used in Section \ref{completing the proof} to show that the sequence of measures $(\mu_k)_{k=0}^{\infty}$ does indeed converges weakly to a measure $\mu$ and to pass from the Fourier decay estimate on $F_M$ to the desired Fourier decay estimate on $\mu$.


\section{Proof of Theorems \ref{main-thm-1} and \ref{main-thm-2}: Notation}\label{sec-notation}

We begin the proof of Theorems \ref{main-thm-1} and \ref{main-thm-2} by clarifying some notation.

Suppose $f : \RR^d \rightarrow \CC$.
If $f \in L^1(\RR^d)$, the Fourier transform of $f$ is defined to be
$$
\widehat{f}(\xi) = \int_{\RR} f(x)e^{-2\pi i x \cdot \xi} dx \quad \forall \xi \in \RR^d.
$$
If $f \in L^1([0,1]^d)$ and $f$ is periodic for the lattice $\ZZ^d$, 
the Fourier transform of $f$ is defined to be
$$
\widehat{f}(\xi) = \int_{[0,1]^d} f(x)e^{-2\pi i x \cdot \xi} dx \quad \forall \xi \in \RR^d.
$$
There is no ambiguity with these definitions; 
if $f \in L^1(\RR^d)$ and $f$ is periodic for the lattice $\ZZ^d$, then $\widehat{f} = 0$ using either definition.


\section{Proof of Theorems \ref{main-thm-1} and \ref{main-thm-2}: The Functions $\phi$ and $\Phi^{\epsilon}_{q,\theta}$}\label{sec-phi-Phi}

In this section, we define the function $\phi$, use it to define the function $\Phi^{\epsilon}_{q,\theta}$, and compute the Fourier transform of $\Phi^{\epsilon}_{q,\theta}$. We will use $\Phi^{\epsilon}_{q,\theta}$ to define the function $F_M$ in Section \ref{key function}.

Let $K$ be a positive integer with $K > mn+a$. Let $\phi: \RR^m \rightarrow \RR$ be a non-negative $C^{K}$ function with $\text{supp}(\phi) \subseteq [-1,1]^m$, and $\int_{\RR^m} \phi(x) dx = 1$. Then there is a $C_1 > 0$ such that
\begin{align}\label{phi-decay}
|\widehat{\phi}(\xi)| \leq C_1 (1+|\xi|)^{-K} \quad \forall \xi \in \RR^m.
\end{align}

For $\epsilon > 0$ and $x \in \RR^m$, let $\phi^{\epsilon}(x) = \epsilon^{-m}\phi(\epsilon^{-1}x)$, and 
$$
\Phi^{\epsilon}(x) = \sum_{k \in \ZZ^m} \phi^{\epsilon}(x-k).
$$
Note $\Phi^{\epsilon}$ is $C^{K}$, periodic for the lattice $\ZZ^m$,
and 
$$
\widehat{\Phi^{\epsilon}}(k) = \widehat{\phi^{\epsilon}}(k) = \widehat{\phi}(\epsilon k) \quad \forall k \in \ZZ^m.
$$
Therefore
\begin{align}\label{F-series-1}
\Phi^{\epsilon}(x) = \sum_{k \in \ZZ^m} \widehat{\phi}(\epsilon k) e^{ 2 \pi i k x } \quad \forall x \in \RR^{m}
\end{align}
with uniform convergence.
%
%

For $q \in \ZZ^n$, $\theta \in \RR^m$, and $x=(x_{11},\ldots,x_{1n},\ldots,x_{m1},\ldots,x_{mn}) \in \RR^{mn}$, define $xq$ by identifying $x$ with the $m \times n$ matrix whose $ij$-entry is $x_{ij}$, and define
\begin{align*}
\Phi^{\epsilon}_{q,\theta}(x) = \Phi^{\epsilon}(xq - \theta). 
\end{align*}
%
%
Note $\Phi^{\epsilon}_{q,\theta}$ is $C^{K}$ and is periodic for the lattice $\ZZ^{mn}$. 
By \eq{F-series-1},
\begin{align*}
\Phi^{\epsilon}_{q,\theta}(x) = \sum_{k \in \ZZ^m} \widehat{\phi}(\epsilon k) e^{ 2 \pi i k \cdot (xq - \theta) } \quad \forall x \in \RR^{mn}
\end{align*}
with uniform convergence.



For $\ell = (\ell_{11},\ldots,\ell_{1n},\ldots,\ell_{m1},\ldots,\ell_{mn}) \in \ZZ^{mn}$, define $\ell_j = (\ell_{1j},\ldots,\ell_{mj})$ for all $1 \leq j \leq n$ and 
$$
D(\ell) = \cbr{q \in \ZZ^n : q_{1}^{-1}\ell_1 = \cdots = q_{n}^{-1}\ell_n \in \ZZ^{m}}.
$$
Note that if $mn=1$, then $D(\ell)$ is the set of all integers that divide $\ell$.

\begin{lemma}\label{Phi-fourier}
For $\ell \in \ZZ^{mn}$ and $q \in \ZZ^n$ with $q_j \neq 0$ for all $1 \leq j \leq n$, 
\begin{align}\label{111}
\widehat{\Phi^{\epsilon}_{q,\theta}}(\ell) 
=
\left\{
\begin{array}{ll}
e^{-2 \pi i q_{1}^{-1} \ell_{1} \cdot \theta} \widehat{\phi}( \epsilon q_{1}^{-1} \ell_{1} ) & \text{if } q \in D(\ell),\\
0 & \text{otherwise. }
\end{array} \right.
\end{align}
\end{lemma}
%
%
\begin{proof}
As a warm-up, note that if $mn = 1$ and $\ell,q \in \ZZ$ with $q \neq 0$ we have
%
\begin{align*}
\widehat{\Phi^{\epsilon}_{q,\theta}}(\ell) 
&=
\sum_{k \in \ZZ}  \int_{[0,1]} \widehat{\phi}(\epsilon k) e^{2 \pi i k (xq - \theta)} e^{- 2 \pi i \ell x} dx \\
\notag
&=
\sum_{k \in \ZZ} e^{- 2 \pi i k \theta } \widehat{\phi}(\epsilon k) \int_{[0,1]} e^{2 \pi i x (kq  - \ell)} dx \\
\notag
&=
\left\{
\begin{array}{ll}
e^{-2 \pi i q^{-1} \ell \theta} \widehat{\phi}( \epsilon q^{-1} \ell ) & \text{if } q \in D(\ell), \\
0 & \text{otherwise. }
\end{array} \right.
\end{align*}

In general, for $\ell = (\ell_{11},\ldots,\ell_{1n},\ldots,\ell_{m1},\ldots,\ell_{mn}) \in \ZZ^{mn}$ and $q \in \ZZ^n$ with $q_j \neq 0$ for all $1 \leq j \leq n$ we have
\begin{align*}
\widehat{\Phi^{\epsilon}_{q,\theta}}(\ell) 
&=
\sum_{k \in \ZZ^m}  \int_{[0,1]^{mn}} \widehat{\phi}(\epsilon k) e^{2 \pi i k \cdot (xq - \theta)} e^{- 2 \pi i \ell \cdot x} dx \\
\notag
&=
\sum_{k \in \ZZ^m}  e^{- 2 \pi i k \cdot \theta } \widehat{\phi}(\epsilon k) \int_{[0,1]^{mn}} e^{2 \pi i (k \cdot (xq - \theta) - \ell \cdot x)} dx \\
\notag
&=
\sum_{k \in \ZZ^m}  e^{- 2 \pi i k \cdot \theta } \widehat{\phi}(\epsilon k) \prod_{i=1}^{m} \prod_{j=1}^{n} \int_{[0,1]} e^{2 \pi i x_{ij}(k_i q_j - \ell_{ij})} dx_{ij} \\
\notag
&=
\left\{
\begin{array}{ll}
e^{-2 \pi i q_{1}^{-1} \ell_{1} \cdot \theta} \widehat{\phi}( \epsilon q_{1}^{-1} \ell_{1} ) & \text{if } q \in D(\ell),\\
0 & \text{otherwise. }
\end{array} \right.
\end{align*}
\end{proof}

\section{Proof of Theorems \ref{main-thm-1} and \ref{main-thm-2}: The Function $F_M$}\label{key function}

In this section, we define the function $F_M$ and discuss some of its properties.

For $M > 0$, define
$$
F_{M}(x) =  \frac{1}{|Q(M)|} \sum_{q \in Q(M)} \Phi^{\epsilon(M)}_{q,\theta}(x) \quad \forall x \in \RR^{mn}.
$$
Note $F_{M}$ is $C^{K}$ and periodic for the lattice $\ZZ^{mn}$.

By the definition of $\Phi^{\epsilon(M)}_{q,\theta}$, we can write
$$
F_{M}(x) = \frac{1}{|Q(M)|} \sum_{q \in Q(M)} \sum_{k \in \ZZ^m} \epsilon(M)^{-m}\phi(\epsilon(M)^{-1}(xq - \theta - k))
\quad \forall x \in \RR^{mn}.
$$
Since $\text{supp}(\phi) \subseteq [-1,1]^{m}$ and $\displaystyle{ \epsilon(M) = \min_{q \in Q(M)} \Psi(q) }$, we have 
\begin{align}\label{support}
\text{supp} (F_{M}) \subseteq \{ x \in \RR^{mn} :  \max_{1 \leq i \leq m} \| \sum_{j=1}^{n} x_{ij}q_j - \theta_i \| \leq \Psi(q) \text{ for some } q \in Q(M) \}.
\end{align}
%
%
By \eq{111}, 
\begin{align}\label{FMhat}
\widehat{F_{M}}(\ell) = \frac{1}{|Q(M)|} \sum_{q \in Q(M) \cap D(\ell)} e^{-2 \pi i q_{1}^{-1} \ell_{1} \cdot \theta} \widehat{\phi}( \epsilon q_{1}^{-1} \ell_{1} )
\quad \forall \ell \in \ZZ^{mn}.
\end{align}
As $\widehat{\phi}(0) = \int_{\RR^m} \phi(x) dx = 1$ and $D(0) = \ZZ^n$, \eq{FMhat} implies
\begin{align}\label{111.9}
\widehat{F_M}(0) = 1.
\end{align}
Since $\phi \geq 0$, we have $F_{M} \geq 0$, and so
\begin{align}\label{112}
|\widehat{F_M}(\ell)| \leq \widehat{F_M}(0) = 1 \quad \forall \ell \in \ZZ^{mn}.
\end{align}
If  $q \in Q(M)$ and $0 < |\ell| \leq M/2$, then for some $j_0 \in \cbr{1,\ldots,n}$ we have $M/2 < |q_{j_0}| \leq M$ and $0 < |\ell_{j_0}| \leq M/2$, hence $0 < |q_{j_0}^{-1}\ell_{j_0}| < 1$. On the other hand, if $q \in D(\ell)$, then $|q_{j}^{-1}\ell_{j}|$ is an integer for all $j \in \cbr{1,\ldots,n}$. Therefore if $q \in Q(M)$ and $0 < |\ell| \leq M/2$, we must have $q \notin D(\ell)$.
So, by \eq{FMhat},
\begin{align} \label{113}
\widehat{F_M}(\ell) = 0 \quad \text{for } 0 < |\ell| \leq M/2.
\end{align}

\section{Proof of Theorems \ref{main-thm-1} and \ref{main-thm-2}: The Fourier Decay of $F_M$}\label{fourier decay FM}

In this section we will prove the following Fourier decay estimate for $F_M$. 

\begin{lemma}\label{lemma-114}
For every $\zeta > \ln 2$ there is an $L_{\zeta} \in \NN$ such that 
\begin{align}\label{114}
|\widehat{F_{M}}(\ell)| \leq C_1 |\ell|^{-a} \exp\rbr{ \frac{\zeta \ln |\ell|}{\ln \ln |\ell|} } h(M) \quad \forall \ell \in \ZZ^{mn}, |\ell| \geq L_{\zeta}, M \in \mathcal{M}.
\end{align}
\end{lemma}

The proof of Lemma \ref{lemma-114} relies on the following divisor bound of Wigert \cite{Wigert} (cf. \cite[p.~262]{HW}).
\begin{lemma}[Wigert]\label{Wigert}
Let $\tau(\ell)$ be the number of positive integer divisors of the integer $\ell$. Then
$$
\limsup_{\ell \rightarrow \infty} \frac{\ln \tau(\ell)}{\ln \ell / \ln \ln \ell} = \ln 2.
$$
\end{lemma}

Besides Wigert's divisor bound, the proof of Lemma \ref{lemma-114} uses the Fourier decay of $\phi$ and the density of $|Q(M)|$.


\begin{proof}[Proof of Lemma \ref{lemma-114}]
Choose $i_0 \in \cbr{1,\ldots,m}$ and $j_0 \in \cbr{1,\ldots,n}$ such that $|\ell| = |\ell_{j_0}|=|\ell_{i_0 j_0}|$. By \eq{phi-decay}, \eq{FMhat}, and the definition of $D(\ell)$, for all $\ell \in \ZZ^{mn}$ we have
\begin{align*}
|\widehat{F_M}(\ell)|
&\leq
C_1 \frac{1}{|Q(M)|} \sum_{q \in Q(M) \cap D(\ell)} 
|\widehat{\phi}( \epsilon(M)  q_{1}^{-1} \ell_1 )| \\
&\leq
C_1 \frac{1}{|Q(M)|} \sum_{q \in Q(M) \cap D(\ell)}  (1 + \epsilon(M)  |q_1|^{-1} |\ell_1| )^{-K}
\\
&=
C_1 \frac{1}{|Q(M)|} \sum_{q \in Q(M) \cap D(\ell)}  (1 + \epsilon(M)  |q_{j_0}|^{-1} |\ell_{j_0}| )^{-K}
\\
&\leq
C_1 \frac{|Q(M) \cap D(\ell)|}{|Q(M)|} (1 + \epsilon(M) M^{-1} |\ell| )^{-K}.
\end{align*} 
We estimate each factor in the last expression separately. Assume $\ell \neq 0$.
Since $K \geq a$, we have
\begin{align*}
(1 + \epsilon(M) M^{-1} |\ell|)^{-K} \leq (\epsilon(M) M^{-1} |\ell|)^{-a}.
\end{align*}
By \eq{main-thm-2 e1}, 
\begin{align*}
|Q(M)|\epsilon(M)^{a} h(M) \geq M^{a} \quad \forall M \in \mathcal{M}.
\end{align*}
To bound $|Q(M) \cap D(\ell)|$, we start with $|Q(M) \cap D(\ell)| \leq |D(\ell)|$. Note
$$
D(\ell)
\subseteq 
\cbr{(q_1,\ldots,q_n) \in \ZZ^n : \frac{\ell_{i_0 j_0}}{q_{j_0}} \in \ZZ, \, q_j = \frac{\ell_{i_0 j} q_{j_0} }{\ell_{i_0 j_0}} \quad \forall 1 \leq j \leq n }.
$$
The set on the right is in bijection with the set of integers that divide $|\ell| = |\ell_{i_0 j_0}|$, so this set has cardinality $2\tau(|\ell|)$ with $\tau$ as in Lemma \ref{Wigert}. Thus 
$$
|Q(M) \cap D(\ell)| \leq |D(\ell)| \leq 2\tau(|\ell|).
$$
It follows from Lemma \ref{Wigert} that for every $\zeta > \ln 2$ there is an $L_{\zeta} \in \NN$ such that 
$$
|Q(M) \cap D(\ell)| \leq \exp\rbr{ \frac{\zeta \ln |\ell|}{\ln \ln |\ell|}  }	 \quad \forall \ell \in \ZZ^{mn}, |\ell| \geq L_{\zeta}, M \in \mathcal{M}.
$$
Putting everything together, we get \eq{114}.
\end{proof}

\section{Proof of Theorems \ref{main-thm-1} and \ref{main-thm-2}: The Key Lemma}\label{main lemma}

In this section, we state and prove the key lemma that will let us pass from the function $F_M$ to the measure $\mu$.
 

Define
$$
g(\xi) = \begin{cases}
1 & \text{if } |\xi| \leq e, \\
|\xi|^{-a} \exp\rbr{ \dfrac{\ln |\xi|}{\ln \ln |\xi|}  }  h(4|\xi|) & \text{if } |\xi| > e.
\end{cases}
$$

\begin{lemma}\label{main-lemma}
For every $\delta > 0$, $M_0 > 0$, and $\chi \in C^{K}_{c}(\RR^{mn})$, there is an $M_{\ast} = M_{\ast}(\delta,M_0,\chi) \in \mathcal{M}$ such that $M_{\ast} \geq M_0$ and
\begin{align*}
|\widehat{\chi F_{M_{\ast}}}(\xi) - \widehat{\chi}(\xi)| \leq \delta g(\xi) \quad \forall \xi \in \RR^{mn}
\end{align*}
\end{lemma}
The proof will show that $M_{\ast}$ can be taken to be any sufficiently large element of $\mathcal{M}$.
%
%
%
%
%
%
\begin{proof}
Since $\chi \in C^{K}_{c}(\RR^{mn})$, there is a $C_2 > 0$ such that
\begin{align}\label{108}
|\widehat{\chi}(\xi)| \leq C_2(1+|\xi|)^{-K} \quad \forall \xi \in \RR^{mn}.
\end{align}
For every $p > mn$, we have
\begin{align}\label{108-2}
\sup_{\xi \in \RR^{mn}} \sum_{\ell \in \ZZ^{mn}} (1+|\xi - \ell|)^{-p} < \infty.
\end{align}

Since $F_M$ is $C^{K}$ and periodic for the lattice $\ZZ^{mn}$, we have
$$
F_M(x) = \sum_{\ell \in \ZZ^{mn}} \widehat{F_M}(\ell) e^{2 \pi i \ell \cdot x} \quad \forall x \in \RR^{mn}
$$
with uniform convergence. 
Since $\chi \in L^1(\RR^{mn})$, multiplying by $\chi$ and taking the Fourier transform yields
\begin{align*}
\widehat{\chi F_M}(\xi) 
= \sum_{\ell \in \ZZ^{mn}} \widehat{F_M}(\ell) \int_{\RR^{mn}} \chi(x) e^{2 \pi i (\ell - \xi) \cdot x} dx
= \sum_{\ell \in \ZZ^{mn}} \widehat{F_M}(\ell) \widehat{\chi}(\xi-\ell)
\end{align*}
for all $\xi \in \RR^{mn}$.
Then by \eq{111.9} and \eq{113} we have
\begin{align}\label{110-2}
\widehat{\chi F_{M}}(\xi) - \widehat{\chi}(\xi)
=
\sum_{\ell \in \ZZ^{mn}} \widehat{\chi}(\xi-\ell) \widehat{F_M}(\ell) - \widehat{\chi}(\xi) 
=
\sum_{|\ell| > M/2} \widehat{\chi}(\xi-\ell) \widehat{F_M}(\ell).
\end{align}%
\textbf{Case 1:} $|\xi| < M/4$. If $|\ell| > M/2$, then 
$|\xi - \ell| > M/4 > |\xi|$. 
Hence by \eq{112}, \eq{108}, \eq{108-2}, \eq{110-2} and because $K > mn + a$ we have
\begin{align*}
| \widehat{\chi F_{M}}(\xi) - \widehat{\chi}(\xi) |
&\leq
C_2 \sum_{|\ell| > M/2} (1+|\xi - \ell|)^{-K} \\
&\leq
C_2 (1 + |\xi|)^{-a} (1 + M/4)^{-(K - a - mn)/2} \sum_{|\ell| > M/2} (1+|\xi - \ell|)^{-mn-(K - a - mn)/2} \\
&\leq
\delta g(\xi)
\end{align*}
for all $M$ sufficiently large.


\textbf{Case 2:} $|\xi| \geq M/4$. Using \eq{110-2}, write 
$$
\widehat{\chi F_{M}}(\xi) - \widehat{\chi}(\xi) 
= \sum_{\substack{|\ell| > M/2 \\ |\ell| \leq |\xi|/2}} \widehat{\chi}(\xi-\ell)  \widehat{F_M}(\ell) 
+ \sum_{\substack{|\ell| > M/2 \\ |\ell| > |\xi|/2}} \widehat{\chi}(\xi-\ell)  \widehat{F_M}(\ell) 
= S_1 + S_2.
$$

If $|\ell| \leq |\xi|/2$, then $|\xi - \ell| \geq |\xi|/2 \geq M/8$. Hence by \eq{112}, \eq{108}, \eq{108-2} and because $K > mn + a$ we have
\begin{align*}
|S_1|
&\leq
C_2 \sum_{\substack{|\ell| > M/2 \\ |\ell| \leq |\xi|/2}} (1+|\xi - \ell|)^{-K} \\
&\leq
C_2 (1 + |\xi|/2)^{-a} (1 + M/8)^{-(K - a - mn)/2} \sum_{\substack{|\ell| > M/2 \\ |\ell| \leq |\xi|/2}} (1+|\xi - \ell|)^{-mn-(K - a - mn)/2} \\
&\leq
\frac{\delta}{2} g(\xi)
\end{align*}
for all $M$ sufficiently large.


Fix $\ln 2 < \zeta < 1$. By \eq{114}, \eq{108}, \eq{108-2} and because $K > mn$ we have
\begin{align*}
|S_2|
&\leq
C_1 C_2 \sum_{\substack{|\ell| > M/2 \\ |\ell| > |\xi|/2}} |\ell|^{-a} \exp\rbr{ \frac{\zeta \ln |\ell|}{\ln \ln |\ell|} } h(M)  (1+|\xi - \ell|)^{-K} \\
&\leq
C_1 C_2 (|\xi|/2)^{-a} \exp\rbr{ \frac{\zeta \ln (|\xi|/2)}{\ln \ln (|\xi|/2)} } h(4|\xi|) \sum_{\substack{|\ell| > M/2 \\ |\ell| > |\xi|/2}}  (1+|\xi - \ell|)^{-K} \\
&\leq \frac{\delta}{2} g(\xi)
\end{align*}
for all sufficiently large $M \in \mathcal{M}$.
\end{proof}

\section{Proof of Theorems \ref{main-thm-1} and \ref{main-thm-2}: The Measure $\mu$}\label{completing the proof}

Let $\chi_0 \in C_{c}^{K}(\RR^{mn})$ with $\int_{\RR^{mn}} \chi_0(x)dx = 1$, $\text{supp}(\chi_0) = [-1,1]^{mn}$, and $\chi_0(x) > 0$ for all $|x|<1$. With the notation of Lemma \ref{main-lemma}, define
$$
M_1 = M_{\ast}(2^{-1},1,\chi_0), \quad M_{k} = M_{\ast}(2^{-k-1},2M_{k-1},\chi_0F_{M_1}\cdots F_{M_{k-1}}) \quad \forall k \geq 2.
$$
Define measures $\mu_k$ by
$$
d\mu_0 = \chi_0 dx, \quad d\mu_k = \chi_0  F_{M_1} \cdots F_{M_{k}} dx \quad \forall k \geq 1.
$$
By Lemma \ref{main-lemma},
\begin{align}\label{7}
|\widehat{\mu_k}(\xi) - \widehat{\mu_{k-1}}(\xi)| \leq 2^{-k-1} g(\xi) \quad \forall \xi \in \RR^{mn}, k \in \NN. 
\end{align}
Since $g(\xi)$ is bounded, \eq{7} implies $(\widehat{\mu_k})_{k = 0}^{\infty}$ is a Cauchy sequence in the supremum norm.
Therefore, since each $\widehat{\mu_k}$ is a continuous function, $\displaystyle{\lim_{k \rightarrow \infty}} \widehat{\mu_k}$ is a continuous function.
By \eq{7}, we have
\begin{align}\label{8}
|\lim_{k \rightarrow \infty} \widehat{\mu_k}(\xi) - \widehat{\mu_0}(\xi)| 
\leq 
\sum_{k=1}^{\infty} |\widehat{\mu_k}(\xi) - \widehat{\mu_{k-1}}(\xi)| 
\leq
g(\xi) \sum_{k=1}^{\infty} 2^{-k-1} = \frac{1}{2}g(\xi)
\end{align}
for all $\xi \in \RR^{mn}$. Since $\widehat{\mu_0}(0) = \int_{\RR^{mn}} \chi_0(x)dx = 1$ and $g(0) = 1$, it follows from \eq{8} that 
$$
1/2 \leq |\displaystyle{\lim_{k \rightarrow \infty}} \widehat{\mu_k}(0)| \leq 3/2.
$$ 
Therefore, by L\'{e}vy's continuity theorem, $(\mu_k)_{k=0}^{\infty}$ converges weakly to a non-trivial finite Borel measure $\mu$ with $\widehat{\mu} = \displaystyle{\lim_{k \rightarrow \infty}} \widehat{\mu_k}$ and
\begin{align*}
\text{supp}(\mu) 
= \bigcap_{k=1}^{\infty} \text{supp}(\mu_k) 
= \text{supp}(\chi_0) \cap \bigcap_{k=1}^{\infty} \text{supp}(F_{M_k}).
\end{align*}
Because $M_{k} \geq 2M_{k-1}$ and because of \eq{support}, we have 
$$\text{supp}(\mu) \subseteq E(m,n,Q,\Psi,\theta).$$
Since $\chi_0 \in C_{c}^{K}(\RR^{mn})$ and $K > a$, we have $\widehat{\mu_0}(\xi) \lesssim (1+|\xi|)^{-a}$ for all $\xi \in \RR^{mn}$. Combining this with \eq{8} gives 
$$
|\widehat{\mu}(\xi)| \lesssim g(\xi) \quad \forall \xi \in \RR^{mn}.
$$
By multiplying $\mu$ by a constant, we can make $\mu$ a probability measure. This completes the proof of Theorems \ref{main-thm-1} and \ref{main-thm-2}.


\section{Proof of Lemma \ref{eta-upper lemma}}\label{appendix}

\begin{lemma}[Restatement of Lemma \ref{eta-upper lemma}]
$$
\dim_{H} E(Q,\Psi,\theta) \leq \min\cbr{\eta(Q,\Psi),1},
$$
where
$$
\eta(Q,\Psi) = \inf \bcbr{\eta \geq 0 : \sum_{\substack{q \in Q \\ q \neq 0}} |q|\rbr{\frac{\Psi(q)}{|q|}}^{\eta} < \infty}.
$$
\end{lemma}
\begin{proof}
Since $\dim_{H} E(Q,\Psi,\theta) \leq \dim_{H} \RR \leq 1$, we only need to prove
$$
\dim_{H} E(Q,\Psi,\theta) \leq \eta(Q,\Psi).
$$
Note $E(Q,\Psi,\theta)$ is invariant under translation by integers. Therefore
\begin{align*}
\dim_{H} E(Q,\Psi,\theta)
&=
\dim_{H} \bigcup_{k \in \ZZ} E(Q,\Psi,\theta) \cap ([0,1] + k) 
=
\dim_{H} \bigcup_{k \in \ZZ} (E(Q,\Psi,\theta) - k) \cap [0,1] \\
&=
\dim_{H} \bigcup_{k \in \ZZ} \dim_{H} E(Q,\Psi,\theta) \cap [0,1] 
=
\dim_{H} E(Q,\Psi,\theta) \cap [0,1].
\end{align*}
So it suffices to prove
$$
\dim_H E(Q,\Psi,\theta) \cap [0,1] \leq \eta(Q,\Psi).
$$
Therefore, according to the definition of Hausdorff dimension, it will suffice to show that for all $\epsilon > 0$ and all $\eta > \eta(Q,\Psi)$ there is a countable collection $\mathcal{I}$ of intervals that covers $E(Q,\Psi,\theta) \cap [0,1]$ and satisfies 
$$
\sum_{I \in \mathcal{I}} (\text{diam} (I))^{\eta} < \epsilon.
$$
Let $\epsilon > 0$ and $\eta > \eta(Q,\Psi)$. Define $C = |\theta| + \displaystyle{\sup_{q \in \ZZ}} \Psi(q)$. 
Observe that 
\begin{align*}
E(Q,\Psi,\theta) \cap [0,1]
&= \bigcap_{N \in \NN} \bigcup_{\substack{ q \in Q \\ |q| \geq N}} \bigcup_{\substack{k \in \ZZ \\ |k| \leq C + |q|}} \cbr{x \in [0,1] : |xq - \theta - k| \leq \Psi(q)}.
\end{align*}
Let $N \in \NN$. Then
\begin{align*}
E(Q,\Psi,\theta) \cap [0,1]
&\subseteq
\bigcup_{\substack{ q \in Q \\ |q| \geq N}} \bigcup_{\substack{k \in \ZZ \\ |k| \leq C + |q|}} \cbr{x \in [0,1] : |xq - \theta - k| \leq \Psi(q)} 
\subseteq \bigcup_{\substack{ q \in Q \\ |q| \geq N}} \bigcup_{\substack{k \in \ZZ \\ |k| \leq C + |q|}} I_{q,k},
\end{align*}
where
$$
I_{q,k} = [(\theta + k - \Psi(q))/|q|, (\theta + k + \Psi(q))/|q|].
$$
We have
\begin{align*}
\sum_{\substack{ q \in Q \\ |q| \geq N}} \sum_{\substack{k \in \ZZ \\ |k| \leq C + |q|}} 
(\text{diam} (I_{q,k}))^{\eta}
&=
\sum_{\substack{ q \in Q \\ |q| \geq N}} \sum_{\substack{k \in \ZZ \\ |k| \leq C + |q|}} 2^{\eta}\rbr{\frac{\Psi(q)}{|q|}}^{\eta} \\
&\leq
\sum_{\substack{ q \in Q \\ |q| \geq N}} (2(C+|q|)+1)2^{\eta}\rbr{\frac{\Psi(q)}{|q|}}^{\eta} \\
&\lesssim
\sum_{\substack{ q \in Q \\ |q| \geq N}} |q|\rbr{\frac{\Psi(q)}{|q|}}^{\eta}.
\end{align*}
The last sum converges because $\eta > \eta(Q,\Psi)$. So, by taking $N$ sufficiently large, we can make the sum less than $\epsilon$.
\end{proof}

\section{Questions for Further Study}\label{Questions for Further Study}

In this section, we pose three questions that are interesting for future research.

What is the Fourier dimension of $E(Q,\Psi_{\tau},0)$ when 
$\nu(Q) 
< 1$? For example, consider $Q$ as the set of squares (so that $\nu(Q) = 1/2$) or the set of powers of $2$ (so that $\nu(Q) = 0$). 
We know the Fourier dimension is at most the Hausdorff dimension $\min\{(1+\nu(Q))/(1+\tau), 1\}$.  
And Theorem \ref{main-thm-1} implies the Fourier dimension is at least $\min\{2\nu(Q)/(1+\tau),1\}$. But when $\nu(Q) < 1$ the exact Fourier dimension is unknown.


What is the Fourier dimension of $E(m,n,\ZZ,\Psi_{\tau},0)$? Theorem \ref{main-thm-2} implies the Fourier dimension is at least $\min\{2n/(1+\tau),mn\}$. It is natural to conjecture that the Fourier dimension is exactly $\min\{2n/(1+\tau),mn\}$.  
It is, perhaps, equally natural conjecture to that $E(m,n,\ZZ,\Psi_{\tau},0)$ is a Salem set, meaning its Fourier dimension is equal to its Hausdorff dimension $\min\{m(n-1) + (m+n))/(1+\tau), mn\}$. The verification of the latter conjecture would make $E(m,n,\ZZ,\Psi_{\tau},0)$ the first explicit example of a Salem set in $\RR^{d}$ $(d \geq 2)$ with dimension strictly between $1$ and $d-1$.

What is the Fourier dimension of $E(m,n,Q,\Psi,\theta)$ when no additional restrictions are placed on the parameters? This is the most general question and therefore the most challenging.

\section{Acknowledgements}\label{Acknowledgements}

The author thanks Izabella {\L}aba for her valuable feedback on this work. The authors thanks the anonymous referee for his/her many valuable comments, including for pointing out that Theorem \ref{mn-app} yields the first explicit examples of sets in $\RR^{d}$ $(d \geq 2)$ with Fourier dimension strictly between $1$ and $d-1$, as discussed in Section \ref{applications}.


\noindent {\sc Kyle Hambrook}, Department of Mathematics, University of British Columbia, Vancouver, BC, V6T1Z2 Canada
                                                                                    
\noindent \texttt{hambrook@math.ubc.ca}

\end{document}